\documentclass[11pt,review]{article}
\usepackage{amsmath,amsthm,amssymb,amscd}
\usepackage{graphicx}
\usepackage{graphics}
\usepackage{verbatim}

\usepackage{mathrsfs}
\usepackage{color}
\usepackage[top=1in, bottom=1in, left=1.25in, right=1.25in]{geometry}
\usepackage{enumerate}

\newtheorem{theorem}{Theorem}[section]

\newtheorem{lemma}[theorem]{Lemma}

\newtheorem{corollary}[theorem]{Corollary}

\begin{document}
	
	\title{Haj\'{o}s-like theorem for signed graphs}
	
	\author{Yingli Kang\thanks{Paderborn Institute for Advanced Studies in
		Computer Science and Engineering and institute for Mathematics,
		Paderborn University,
		Warburger Str. 100,
		33102 Paderborn,
		Germany; yingli@mail.upb.de; 
		Fellow of the International Graduate School ``Dynamic Intelligent Systems''.}}
	
	\date{}
	
\maketitle

\begin{abstract}
The paper designs five graph operations, and proves that every signed graph with chromatic number $q$ can be obtained from all-positive complete graphs $(K_q,+)$ by repeatedly applying these operations. This result gives a signed version of the Haj\'{o}s theorem, emphasizing the role of all-positive complete graphs played in the class of signed graphs, as played in the class of unsigned graphs.
\end{abstract}

\section{Introduction}
 We consider a graph to be finite and simple, i.e., with no loops or multiple edges.
 Let $G$ be a graph and $\sigma\colon\ E(G)\rightarrow \{1,-1\}$ be a mapping. The pair $(G,\sigma)$ is called a \emph{signed graph}. We say that $G$ is the \emph{underlying graph} of $(G,\sigma)$ and $\sigma$ is a \emph{signature} of $G$.
 The \emph{sign} of an edge $e$ is the value of $\sigma(e)$, and the \emph{sign product} $sp(H)$ of a subgraph $H$ is defined as $sp(H)=\prod_{e\in E(H)}\sigma(e)$.
 An edge is \emph{positive} if it has positive sign; otherwise, the edge is \emph{negative}.
 A signature $\sigma$ is \emph{all-positive} (resp., \emph{all-negative}) if it has positive sign (resp., negative sign) on each edge.
A graph $G$ together with an all-positive signature is denoted by $(G,+)$ and similarly, $(G,-)$ denotes a signed graph where the signature is all-negative.
Throughout the paper, to be distinguished from ``a signed graph'' and from ``a multigraph'', ``a graph'' is always regarded as an unsigned simple graph.

Let $(G,\sigma)$ be a signed graph. For $v \in V(G)$, denote by $E(v)$ the set of edges incident to $v$.
A \emph{switching} at a vertex $v$ defines a signed graph $(G,\sigma')$ with $\sigma'(e) = -\sigma(e)$ if $e \in E(v)$, and $\sigma'(e) = \sigma(e)$ if $e \in E(G)\setminus E(v)$.
Two signed graphs $(G,\sigma)$ and $(G,\sigma^*)$ are {\em switch-equivalent} (briefly, \emph{equivalent}) if they can be obtained from each other by a sequence of switchings. We also say that
$\sigma$ and $\sigma^*$ are \emph{equivalent signatures} of $G$.

A signed graph $(G,\sigma)$ is \emph{balanced} if each circuit contains even number of negative edges; otherwise, $(G,\sigma)$ is \emph{unbalanced}.
A signed graph $(G,\sigma)$ is \emph{antibalanced} if each circuit contains even number of positive edges.
It is well known (see e.g. \cite{Raspaud_2011}) that $(G,\sigma)$ is balanced if and
only if $\sigma$ is switch-equivalent to an all-positive signature, and $(G,\sigma)$ is antibalanced
if and only if $\sigma$ is switch-equivalent to an all-negative signature.

In the 1980s, Zaslavsky \cite{Zaslavsky_1982, Zaslavsky_1984} initiated the study on vertex colorings of signed graphs.
The natural constraints for a coloring $c$ of a signed graph $(G,\sigma)$ are that, (1) $c(v) \not= \sigma(e) c(w)$ for each edge $e=vw$, and (2) the colors can be inverted under switching, i.e., equivalent signed graphs have the same chromatic number.
In order to guarantee these properties of a coloring, Zaslavsky \cite{Zaslavsky_1982} used 
$2k+1$ ``signed colors" from the color set $\{-k, \dots, 0, \dots, k\}$ and studied the
interplay between colorings and zero-free colorings through the chromatic polynomial.

Recently, M\'a\v{c}ajov\'a, Raspaud and \v{S}koviera \cite{Raspaud_2014}
modified this approach. For $n = 2k+1$ let $M_n = \{0, \pm 1, \dots,\pm k\}$, and
for $n = 2k$ let $M_n = \{\pm 1, \dots,\pm k\}$.
A mapping $c$ from $V(G)$ to $M_n$ is a {\em signed $n$-coloring} of $(G,\sigma)$, if $c(v) \not= \sigma(e) c(w)$ for each edge $e=vw$.
They  defined $\chi_{\pm}((G,\sigma))$ to be the smallest number $n$ such that $(G,\sigma)$ has a signed $n$-coloring, and called it the \textit{signed chromatic number}.
A distinct version of vertex colorings of signed graphs, defined by homomorphisms of signed graphs, was proposed in \cite{Edita_Sopena_2014}.

In \cite{KS_2015}, the authors studied circular colorings of signed graphs. The related integer $k$-coloring of a signed graph $(G,\sigma)$ is defined as follows.
Let $\mathbb{Z}_k$ denote the cyclic group of integers modulo $k$, and the inverse of an element $x$ is denoted by $-x$.
A function $c : V(G) \rightarrow \mathbb{Z}_k$ is a \emph{$k$-coloring} of $(G,\sigma)$ if $c(v) \not= \sigma(e) c(w)$ for each edge $e=vw$. Clearly,
such colorings satisfy the constrains (1) and (2) of a vertex coloring of signed graphs.
The {\em chromatic number $\chi((G,\sigma))$} of a signed graph $(G,\sigma)$ is the smallest $k$ such that $(G,\sigma)$ has a $k$-coloring.
As shown in \cite{KS_2015}, two equivalent signed graphs have the same chromatic number.
In this paper, we follow this version of vertex colorings of signed graphs. 

Many questions concerning the colorings of a signed graph have been discussed. In \cite{Raspaud_2014} and \cite{Stiebitz_2015},
they study the signed chromatic number $\chi_{\pm}$ of signed graphs.
The chromatic spectrum and signed chromatic spectrum of signed graphs are given in \cite{Yingli_2015_00614}.
A few classical results concerning the choice number of graphs are generalized to signed graphs in \cite{Steffen_2015}.
This paper addresses an analogue of a well-known theorem of Haj\'{o}s for signed graphs.

In 1961, Haj\'{o}s proved a result on the chromatic number of graphs, which is one of the classical results in the field of graph colorings. This result has several equivalent formulations, one of them states as the following two theorems.
\begin{theorem}[\cite{Hajos_1961}] \label{thm_closed}
	The class of all graphs that are not $q$-colorable is closed under the following three operations:
	\begin{enumerate}[(1)]
		\setlength{\itemsep}{-0.1cm}
		\item Add vertices or edges;
		\item Identify two nonadjacent vertices;
		\item Let $G_1$ and $G_2$ be two vertex-disjoint graphs with $a_1b_1\in E(G_1)$ and $a_2b_2\in E(G_2)$. Make a graph $G$ from $G_1\cup G_2$ by removing $a_1b_1$ and $a_2b_2$, identifying $a_1$ with $a_2$, and adding a new edge between $b_1$ and $b_2$ (see Figure \ref{oper_3}).
	\end{enumerate}	
\end{theorem}

\begin{figure}[h]
	\centering
	\includegraphics[width=12cm]{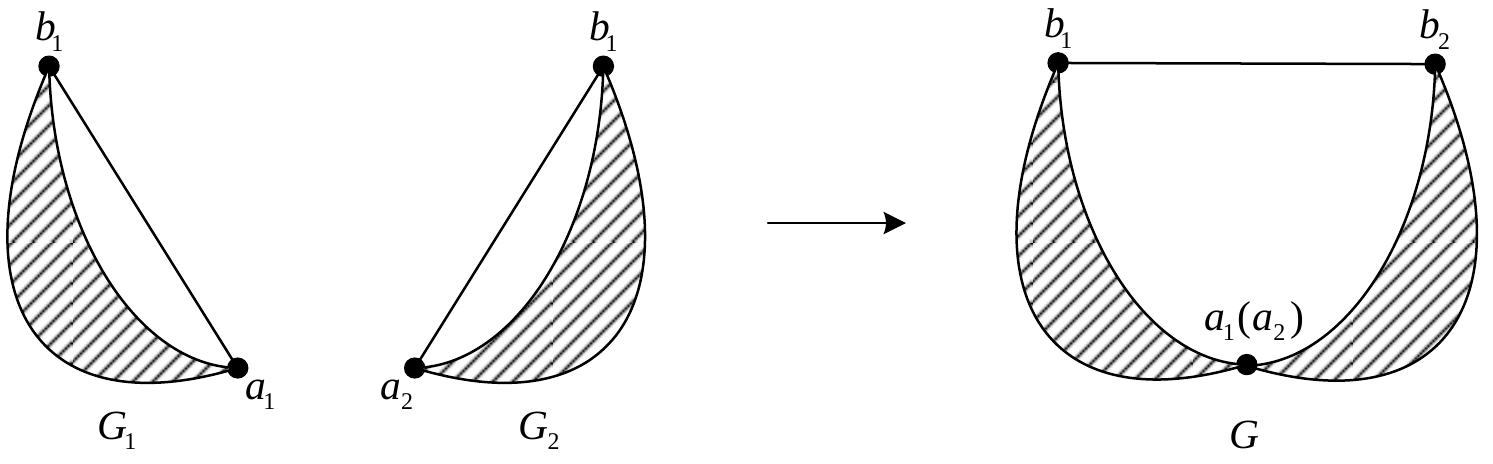} \\
	\caption{Operation $(3)$} \label{oper_3}
\end{figure}

The Operation (3) is known as the Haj\'{o}s construction in literature.
\begin{theorem}[Haj\'{o}s theorem, \cite{Hajos_1961}] \label{thm_Hajos}
	Every non-$q$-colorable graph can be obtained by Operations (1)-(3) from the complete graph $K_{q+1}$.
\end{theorem}

The Haj\'{o}s theorem have been generalized in several different ways, by considering more general colorings than vertex $k$-colorings of graphs. 
The analogues of the Haj\'{o}s theorem are proposed for list-colorings \cite{Gravier_1996}, for weighted colorings \cite{Araujo_2013}, and for group colorings \cite{An_2010}. 
However, all of these extensions are still restricted to unsigned graphs. 

In this paper, we analogously establish a result on the chromatic number $\chi$ of signed graphs, that generalizes the result of Haj\'{o}s to signed graphs. Hence, this result is a signed version of the Haj\'{o}s theorem and is called the Haj\'{o}s-like theorem of signed graphs (briefly, the Haj\'{o}s-like theorem). 
To prove this theorem, we consider signed multigraphs rather than signed simple graphs.
Indeed, for vertex colorings of signed multigraphs, it suffices to consider signed bi-graphs, a subclass of signed multigraphs in which no two edges of the same sign locate between two same vertices.
Clearly, signed bi-graphs contains signed simple graphs as a particular subclass.
Hence, the Haj\'{o}s-like theorem holds for signed bi-graphs and in particular, for signed graphs.
Moreover, the theorem shows that, for the class of signed bi-graphs, the complete graphs together with an all-positive signature plays the same role as it plays for the class of unsigned graphs.

The structure of the rest of the paper are arranged as follows.
In section \ref{sec_complete}, we design five operations on signed bi-graphs and show that these operations are closed in the class of non-$q$-colorable signed bi-graphs for any given positive integer $q$.
Moreover, we established some lemmas necessary for the proof of the Haj\'{o}s-like theorem.
In section \ref{sec_Hajos}, we address the proof of the Haj\'{o}s-like theorem.

\section{Graph operations on signed bi-graphs} \label{sec_complete}
\subsection{Signed bi-graphs}
A \emph{bi-graph} is a multigraph having no loops and having at most two edges between any two distinct vertices.
Let $G$ be a bi-graph, $u$ and $v$ be two distinct vertices of $G$.
Denote by $E(u,v)$ the set of edges connecting $u$ to $v$, and let $m(u,v)=|E(u,v)|$.
Clearly, $0 \leq m(u,v)\leq 2$.
A bi-graph $G$ is \emph{bi-complete} if $m(x,y)=2$ for any $x,y \in V(G)$, and is \emph{just-complete} if  $m(x,y)=1$ for any $x,y \in V(G)$.
A \emph{signed bi-graph} $(G,\sigma)$ is a bi-graph $G$ together with a signature $\sigma$ of $G$ such that any two multiple edges have distinct signs. 
A bi-complete signed bi-graph of order $n$ is denoted by $(K_n,\pm)$.
It is not hard to see that $\chi((K_n,\pm))=2n-2$ and $\chi((K_n,+))=n$.
The concepts of $k$-coloring, chromatic number and switching of signed graphs are naturally extended to signed bi-graphs, working in the same way, and the related notations are inherited.

Let $(G,\sigma)$ be a signed multigraph. Between each pair of vertices, remove all the multiple edges of the same sign but one. We thereby obtain a signed bi-graph $(G',\sigma')$. We can see that $c$ is a $k$-coloring of $(G,\sigma)$ if and only if $c'$ is a $k$-coloring of $(G',\sigma')$, where $c'$ is the restriction of $c$ into $(G',\sigma')$.  
Therefore, for the vertex colorings of signed multigraphs, it suffices to consider signed bi-graphs.

\subsection{Graph operations}
Let $k$ be a nonnegative integer. A signed bi-graph is \emph{$k$-thin} if it is a bi-complete signed bi-graph minus at most $k$ pairwise vertex-disjoint edges. Clearly, if a signed bi-graph is $0$-thin, then it is bi-complete.
\begin{theorem} \label{thm_closed_sb}
The class of all signed bi-graphs that are not $q$-colorable is closed under the following operations:
\begin{enumerate}[(sb1)]
\setlength{\itemsep}{-0.1cm}
  \item Add vertices or signed edges.
  \item Identify two nonadjacent vertices.
  \item Let $(G_1,\sigma_1)$ and $(G_2,\sigma_2)$ be two vertex-disjoint signed bi-graphs. Let $v$ be a vertex of $G_1$ and $e$ be a positive edge of $G_2$ with ends $x$ and $y$. Make a graph $(G,\sigma)$ from $(G_1,\sigma_1)$ and $(G_2,\sigma_2)$ by splitting $v$ into two new vertices $v_1$ and $v_2$, removing $e$, and identifying $v_1$ with $x$ and $v_2$ with $y$ (see Figure  \ref{oper_sb3}).
  \item Switch at a vertex.
  \item When $q$ is even, remove a vertex that has at most $\frac{q}{2}$ neighbors; when $q$ is odd, remove a negative edge whose ends are connected by no other edges, identify these two ends, and add signed edges so that the resulting bi-signed graph is
  $\frac{q-3}{2}$-thin.

\end{enumerate}
  \begin{figure}[h]
  	\centering
  	\includegraphics[width=12cm]{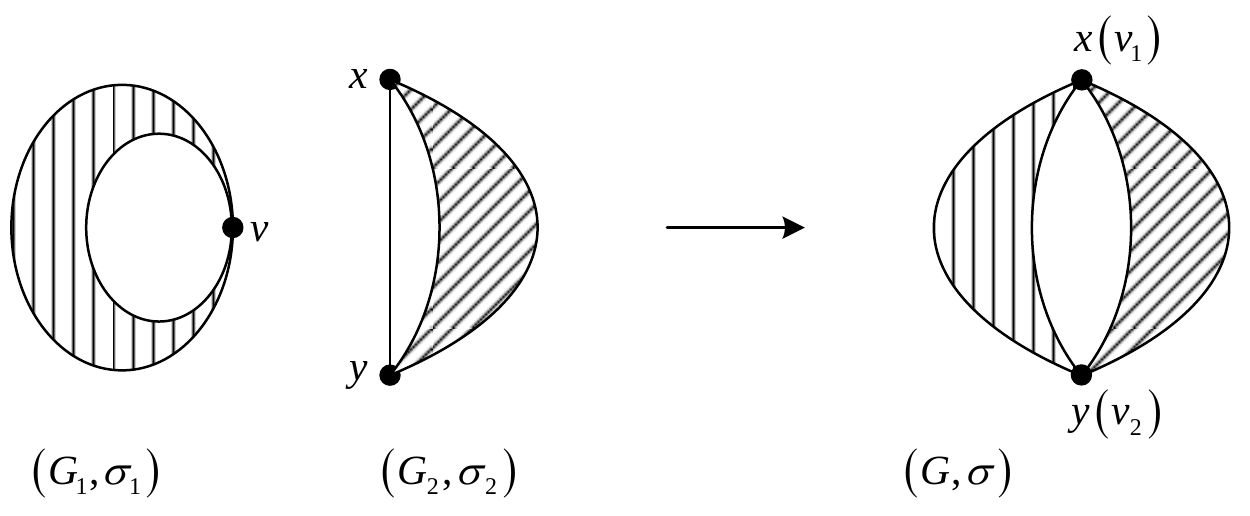} \\
  	\caption{Operation $(sb3)$} \label{oper_sb3}
  \end{figure}
\end{theorem}

\begin{proof}
Since Operations $(sb1),(sb2),(sb4)$ neither make loops nor decrease the chromatic number, it follows that the class of non-$q$-colorable signed bi-graphs is closed under these operations.

For Operation $(sb3)$, suppose to the contrary that $(G,\sigma)$ is $q$-colorable.
Let $c$ be a $q$-coloring of $(G,\sigma)$.
Denote by $x'$ and $y'$ the vertices of $G$ obtained from $x$ and $y$, respectively.
If $c(x')=c(y')$, then the restriction of $c$ into $G_1$, where $v$ is assigned with the same color as $x'$ and $y'$, gives a $q$-coloring of $(G_1,\sigma_1)$, contradicting with the fact that $(G_1,\sigma_1)$ is not $q$-colorable.
Hence, we may assume that $c(x')\neq c(y')$.
Note that $e$ is a positive edge of $(G_2,\sigma_2)$.
Thus the restriction of $c$ into $G_2$ gives a $q$-coloring of $(G_2,\sigma_2)$, contradicting with the fact that $(G_2,\sigma_2)$ is not $q$-colorable.
Therefore, the statement holds for Operation $(sb3)$.

It remains to verify the theorem for Operation $(sb5)$.
For $q$ even, suppose to the contrary that the removal of a vertex $u$ from a non-$q$-colorable signed bi-graph $(G,\sigma)$ yields the $q$-colorability. Let $\phi$ be a $q$-coloring of $(G,\sigma)-u$ using colors from a set $S$, where $S=\{0,\pm 1,\ldots,\pm (\frac{q}{2}-1),\frac{q}{2}\}$.
Notice that each neighbor of $u$ makes at most two colors unavailable for $u$. Since $u$ has at most $\frac{q}{2}$ neighbors, $S$ still has a color available for $u$.
Hence, we can extend $\phi$ to be a $q$-coloring of $(G,\sigma)$, a contradiction.

For the case that $q$ is odd, let $(H,\sigma_H)$ be obtained from a non-$q$-colorable signed bi-graph $(H',\sigma_H')$ by applying this operation to a negative edge $e'$. Let $z$ be the resulting vertex from the two ends of $e'$, say $x'$ and $y'$.
Suppose to the contrary that $(H,\sigma_H)$ is $q$-colorable. Let $\psi$ be a $q$-coloring of $(H,\sigma_H)$ using colors from the set $\{0,\pm 1,\ldots,\pm (\frac{q-1}{2})\}$.
If $\psi(z)\neq 0$, then by assigning $x'$ and $y'$ with the color $\psi(z)$, we complete a $q$-coloring of  $(H',\sigma_H')$, a contradiction.
Hence, we may assume that $\psi(z)=0$.
For $0\leq i\leq \frac{q-1}{2}$, let $V_i=\{v\in V(G)\colon\ |\psi(v)|=i\}$.
Clearly, each $V_i$ is an antibalanced set and in particular, $V_0$ is an independent set.
Since $(H,\sigma_H)$ is $\frac{q-3}{2}$-thin, we can deduce that there exists $p\in \{1,\ldots,\frac{q-1}{2}\}$ such that $|V_p|=1$.
Exchange the colors between $V_0$ and $V_p$, and then assign $x'$ and $y'$ with the same color as $z$, we thereby obtain a $q$-coloring of $(H',\sigma_H')$ from $\psi$, a contradiction.
\end{proof}
\subsection{Useful lemmas}
Operation $(sb3)$ can be extended to the following one which works for signed bi-graphs.
\begin{enumerate}[(sb3')]
	\item Let $(G_1,\sigma_1)$ and $(G_2,\sigma_2)$ be two vertex-disjoint signed bi-graphs.
	For each $i\in \{1,2\}$, let $e_i$ be an edge of $G_i$ with ends $x_i$ and $y_i$.
	Make a graph $(G,\sigma)$ from $G_1\cup G_2$ by removing $e_1$ and $e_2$, identifying $x_1$	with $x_2$, and adding a new edge $e$ between $y_1$ and $y_2$ with $\sigma(e)=\sigma_1(e_1)\sigma_2(e_2)$.
\end{enumerate}

\begin{lemma} \label{lem_comb}
	Operation $(sb3')$ is a combination of Operations $(sb3)$ and $(sb4)$.
\end{lemma}

\begin{proof}
	We use the notations in the statement of Operation $(sb3')$.
	First assume that at least one of $e_1$ and $e_2$ is a positive edge.
	With loss of generality, say $e_1$ is positive.
	We apply Operation $(sb3)$ to $(G_1,\sigma_1)$ and $(G_2,\sigma_2)$ where $e_1$ is removed, $x_2$ is split into two new vertices $x_2'$ and $x_2''$ with $y_2$ as the neighbor of $x_2'$ and all other neighbors of $x_2$ as the neighbors of $x_2''$, and then $x_2'$ is identified with $y_1$ and $x_2''$ is identified with $x_1$. The resulting signed bi-graph is exactly $(G,\sigma)$, we are done.
	Hence, we may next assume that both $e_1$ and $e_2$ are negative edges.
	Switch at $x_1$ in $(G_1,\sigma_1)$ and at $x_2$ in $(G_2,\sigma_2)$.
	Since $e_1$ and $e_2$ are positive in the resulting signed bi-graphs, we may apply Operation $(sb3)$ similarly as above,
	obtaining a signed bi-graph, which leads to $(G,\sigma)$ by switching again at $x_1$ and $x_2$.
\end{proof}

\begin{lemma}\label{lem_antibalanced}
	A just-complete signed bi-graph is antibalanced if and only if the sign product on each triangle is $-1$,
	and is balanced if and only if the sign product on each triangle is $1$.
\end{lemma}

\begin{proof}
	For the first statement, since a just-complete signed bi-graph $(G,\sigma)$ is exactly a complete signed graph, $G$ is antibalanced if and only if the sign product on each circuit of length $k$ is $(-1)^k$. Hence, the proof for the necessity is trivial. Let us proceed to the sufficiency, which will be proved by induction on $k$.
	
	Clearly, the statement holds for $k=3$ because of the assumption of the lemma. Assume that $k\geq 4$. Let $C$ be a circuit of length $k$.
	Take any chord $e$ of $C$, which divides $C$ into two circuits $C_1$ and $C_2$. For $i\in \{1,2\}$, let $k_i$ denote the length of $C_i$. Thus, $k=k_1+k_2-2.$
	By applying the induction hypothesis, we have $sp(C_i)=(-1)^{k_i}$. It follows that $sp(C)=sp(C_1)sp(C_2)=(-1)^k,$ the statement also holds.
	
	The second statement can be argued in the same way as for the first one.
	We only have to pay attention to the equivalence between that $(G,\sigma)$ is balanced and that the sign product on each circuit of length $r$ is 1.
\end{proof}

A signed bi-graph of order $3r$ is \emph{$\triangledown$-complete} if it is $(K_{3r},\pm)$ minus $r$ pairwise vertex-disjoint all-positive triangles.
Clearly, a $\triangledown$-complete signed bi-graph is complete.

\begin{lemma} \label{lem_triangle-complete}
	The $\triangledown$-complete signed bi-graph of order $3r$ can be obtained from $(K_{2r+1},+)$ by Operations (sb1)-(sb5).
\end{lemma}
\begin{proof}
	Take $r+1$ copies of $(K_{2r+1},+)$, say $(H_i,+)$ of vertex set $\{v_i^0,\ldots,v_i^{2r}\}$ for $0\leq i\leq r$.
	For each $j\in \{1,\ldots,r\}$, switch at $v_0^j$, and then apply Operation (sb3') to $H_0$ and $H_j$ so that $v_0^jv_0^{j+r}$ and $v_j^0v_j^{2j}$ are removed and that $v_0^j$ is identified with $v_j^0$, and finally identify $v_0^j$ with $v_0^{j+r}$.
	The resulting signed bi-graph is denoted by $(G,\sigma)$.
	By Theorem \ref{thm_closed_sb}, since $(K_{2r+1},+)$ is not $2r$-colorable,
	$(G,\sigma)$ is not $2r$-colorable either.
	Note that $v_0^0$ has precisely $r$ neighbors in $G$. We can apply Operation (sb5) to $v_0^0$, i.e., we remove $v_0^0$ from $(G,\sigma)$.
	In the resulting signed bi-graph, for each $1\leq k\leq 2r$,
	since $v_1^k,\ldots,v_r^k$ are pairwise nonadjacent, we can apply Operation (sb1) to identify them into one vertex.
	Denote by $(H,\sigma_H)$ the resulting signed bi-graph.
	
	We can see that $(H,\sigma_H)$ is of order $3r$ and moreover, for $1\leq j\leq r$, the signed bi-graph induced by $\{v_0^j,v_1^{2j},v_1^{2j-1}\}$ is an unbalanced triangle.
	It follows that, by adding signed edges and switching if needed, we obtain the $\triangledown$-complete signed bi-graph of order $3r$ from $(H,\sigma_H)$.
\end{proof}

\begin{lemma}\label{lem_bicomplete}
	$(K_r,\pm)$ can be obtained from $(K_{2r-2},+)$ by Operation (sb1)-(sb5).
\end{lemma}
\begin{proof}
	Let $(G,\sigma)$ be a copy of $(K_{2r-2},+)$ of vertices $v_1,\ldots,v_{2r-2}$.
	Clearly, $(G,\sigma)$ is not $(2r-3)$-colorable.
	Switch at $v_1$ and apply Operation (sb5) to $v_1v_2$ so that each of $v_3v_4,v_5v_6,\ldots,v_{2r-5}v_{2r-4}$ has no multiple edges.
	For each $i\in \{2,3,\cdots,r-2\}$, switch at $v_{2i}$ and apply Operation (sb5) to $v_{2i-1}v_{2i}$ so that no new signed edges are added.
	The resulting signed bi-graph is exactly $(K_r,\pm)$.
\end{proof}

\section{Haj\'{o}s-like theorem}\label{sec_Hajos}
We will need the following definitions for the proof of the Haj\'{o}s-like theorem.

Let $(G,\sigma)$ be a signed bi-graph. 
An \emph{antibalanced set} is a set of vertices that induce an antibalanced graph. 
Let $c$ be a $k$-coloring of $(G,\sigma)$. A set of all vertices $v$ with the same value of $|c(v)|$ is called a \emph{partite set} of $(G,\sigma)$. Let $U$ and $V$ be two partite sets. They are \emph{completely adjacent} if $m(u,v)\geq 1$ for any $u\in U$ and $v\in V$, \emph{bi-completely adjacent} if $m(u,v)=2$ for any $u\in U$ and $v\in V$, and \emph{just-completely adjacent} if $m(u,v)=1$ for any $u\in U$ and $v\in V$.

Let $(G,\sigma)$ be a signed bi-graph. A sequence $(x,y,z)$ of three vertices of $G$ is a \emph{triple} if there exist three integers $a,b,c$ satisfying
the following three conditions:
\begin{enumerate}[(i)]
	\setlength{\itemsep}{-0.1cm}
	\item $a,b,c\in \{1,-1\}$,
	\item $ab=c,$
	\item $a\notin \{\sigma(e)\colon\ e\in E(x,y)\}$, $b\notin \{\sigma(e)\colon\ e\in E(x,z)\}$, and $c\in \{\sigma(e)\colon\ e\in E(y,z)\}$.
\end{enumerate}
The sequence $(a,b,c)$ is called a \emph{code} of $(x,y,z)$.
Note that a triple may have more than one code.

\begin{theorem} \label{thm_Hajos_sb} (Haj\'{o}s-like theorem)
Every signed bi-graph with chromatic number $q$ can be obtained from $(K_q,+)$ by Operations (sb1)-(sb5).
\end{theorem}

\begin{proof}
Let $(G,\sigma)$ be a counterexample with minimum $|V(G)|$ and subjecting to it, $|E(G)|$ is maximum.

We first claim that $(G,\sigma)$ is complete.
Suppose to the contrary that $G$ has two non-adjacent vertices $x$ and $y$.
Let $(G_1,\sigma_1)$ and $(G_2,\sigma_2)$ be obtained from a copy of $(G,\sigma)$ by identifying $x$ with $y$ into a new vertex $v$ and by adding a positive edge $e$ between $x$ and $y$, respectively.
Since $(G,\sigma)$ has chromatic number $q$, it follows with Theorem \ref{thm_closed_sb} that both $(G_1,\sigma_1)$ and $(G_2,\sigma_2)$ has chromatic number at least $q$. Note the fact that $(K_i,+)$ can be obtained from $(K_j,+)$ by Operation (sb1) whenever $i>j$.
Thus by the minimality of $|V(G)|$, the graph $(G_1,\sigma_1)$ can be obtained from $(K_q,+)$ by Operations (sb1)-(sb5), and by the maximality of $|E(G)|$, so does $(G_2,\sigma_2)$.
We next show that $(G,\sigma)$ can be obtained from $(G_1,\sigma_1)$ and $(G_2,\sigma_2)$ by Operations $(sb2)$ and $(sb3)$, which contradicts the fact that $(G,\sigma)$ is a counterexample.
This contradiction completes the proof of the claim.
Apply Operation $(sb3)$ to $(G_1,\sigma_1)$ and $(G_2,\sigma_2)$ so that $e$ is removed and $v$ is split into $x$ and $y$.
In the resulting graph, identify each pair of vertices that corresponds to the same vertex of $G$ except $x$ and $y$, we thereby obtain exactly $(G,\sigma)$.

We next claim that $(G,\sigma)$ has no triples. The proof of this claim is analogous to the one above. Suppose to the contrary that $(G,\sigma)$ has a triple, say $(x,y,z)$. Let $(a,b,c)$ be a code of $(x,y,z)$.
Take two copies of $(G,\sigma)$.
Add an edge $e_1$ with sign $a$ into one copy between $x$ and $y$, obtaining $(G',\sigma')$.
Add an edge $e_2$ with sign $b$ into the other copy between $x$ and $z$, obtaining $(G'',\sigma'')$.
Clearly, both $(G',\sigma')$ and $(G'',\sigma'')$ have chromatic number at least $q$. By the maximality of $|E(G)|$, they can be obtained by Operations (sb1)-(sb5) from $(K_q,+)$.
To complete the proof of the claim, it remains to show that $(G,\sigma)$ can be obtained from $(G',\sigma')$ and $(G'',\sigma'')$ by Operations (sb1)-(sb5).
Note that Operations (sb3') is a combination of Operations $(sb3)$ and $(sb4)$ by Lemma \ref{lem_comb}.
Apply Operation (sb3') to $(G',\sigma')$ and $(G'',\sigma'')$ so that $e_1$ and $e_2$ are removed, $x'$ is identified with $x''$, and an edge $e$ is added between $y'$ and $z''$.
We have $\sigma(e)=\sigma(e_1)\sigma(e_2)=ab=c\in E(y,z)$.
By applying Operation (sb2) to each pair of vertices that are the copies of the same vertex of $G$ except $x$, we obtain $(G,\sigma)$.

We proceed the proof of the theorem by distinguishing two cases according to the parity of $q$.

Case 1: $q$ is odd.
Since $\chi((G,\sigma))=q$, the vertex set $V(G)$ can be divided into $k$ partite sets $V_1,\ldots,V_{k}$, where $k=\frac{q+1}{2}$, so that $V_1$ is an independent set and all others are antibalanced sets but not independent.
It follows that $|V_i|\geq 2$ for all $i\in \{2,\ldots,k\}$. By the first claim, $|V_1|=1$, and the graphs induced by $V_2,\dots,V_k$ are just-complete and moreover,
every two partite sets are completely adjacent.

Subcase 1.1: every two partite sets are bi-completely adjacent.
Take the vertex in $V_1$ and two arbitrary vertices from each of $V_2,\ldots,V_k$.
Let $(H,\sigma_H)$ be the signed bi-graph induced by these vertices.
Clearly, $|V(H)|=q.$
By the first claim and the assumption, we can see that $(H,\sigma_H)$ is a bi-complete signed bi-graph minus disjoint edges.
Hence, $(H,\sigma_H)$ can be obtained from $(K_q,+)$ by switching at vertices and adding signed edges.
Therefore, $(G,\sigma)$ can be obtained from $(K_q,+)$ by Operations $(sb1)$ and $(sb4)$, a contradiction.

Subcase 1.2: there exist two partite sets $V_j$ and $V_l$ that are not bi-completely adjacent.
If $V_j$ and $V_l$ are not just-completely adjacent, then there always exist three vertices $x,y,z$, w.l.o.g., say $x\in V_j$ and $y,z\in V_l$, such that $m(x,y)=1$ and $m(x,z)=2.$
Note that $m(y,z)=1$. Thus, $(y,x,z)$ is a triple of $(G,\sigma)$, contradicting the second claim.
Hence, we may assume that $V_j$ and $V_l$ are just-completely adjacent.
Recall that both  $V_j$ and $V_l$ induce just-complete signed bi-graphs.
Thus, $V_j\cup V_l$ induces a just-complete signed bi-graphs as well, say $(Q,\sigma_Q)$.
By again the second claim, every triangle in $(Q,\sigma_Q)$ has sign product -1.
Thus, $(Q,\sigma_Q)$ is antibalanced by Lemma \ref{lem_antibalanced}.
It follows that $1\in \{j,l\}$ since otherwise, the division of $V(G)$, obtained from $\{V_1,\ldots,V_{k}\}$ by merging $V_j$ with $V_l$,
yields $\chi((G,\sigma))\leq q-2$, a contradiction.
W.l.o.g., let $j=1$.

We next show that every other pair of partite sets are bi-completely adjacent.
Suppose to the contrary that there exist another two partite sets, say $V_s$ and $V_t$, that are not bi-completely adjacent.
By the same argument as above, $1\in\{s,t\}$. We may assume that $s=1$.
Let $u_1,u_l,u_t$ be a vertex of $V_1,V_l,V_t$, respectively, such that $m(u_1,u_l),m(u_1,u_t)\leq 1.$
Note that $V_l$ and $V_t$ are bi-completely adjacent, we have $m(u_l,u_t)=2.$
It follows that $(u_1,u_l,u_t)$ is a triple of $(G,\sigma)$, contradicting the second claim.

Recall that $V_j\cup V_l$ is an antibalanced set. It follows that, except $V_j$ and $V_l$, any other partite set contains at least 3 vertices since otherwise,
say $|V_r|=2$ with $r\notin \{j,l\}$, the division of $V(G)$, obtained from $\{V_1,\ldots,V_{k}\}$ by merging $V_j$ with $V_l$ and splitting $V_r$ into two independent sets, yields $\chi((G,\sigma))\leq q-1$, a contradiction.

Take a vertex from $V_j$, two vertices from $V_l$ and three vertices from each of the rest partite sets.
Denote by $(H,\sigma_H)$ the signed bi-graph induced by these vertices. Clearly, $|V(H)|=\frac{3(q-1)}{2}$.
By mutiplicity of the edges in $(G,\sigma)$, we can see that $(H,\sigma_H)$ is a
$\triangledown$-complete signed bi-graph. By Lemma \ref{lem_triangle-complete}, $(H,\sigma_H)$ can be obtained from $(K_q,+)$ by Operations (sb1)-(sb5) and therefore, so does $(G,\sigma)$, a contradiction.

Case 2: $q$ is even.
Since $\chi((G,\sigma))=q$, the vertex set $V(G)$ can be divided into $k$ partite sets $V_1,\ldots,V_{k}$, where $k=\frac{q+2}{2}$, so that at least two of them are independent sets, say $V_1$ and $V_2$.

Subcase 2.1: every two partite sets are bi-completely adjacent.
Take a vertex from each partite set.
Clearly, these vertices induce $(K_\frac{q+2}{2},\pm)$.
Hence, $(G,\sigma)$ can be obtained from $(K_\frac{q+2}{2},\pm)$ by Operation (sb1).
By Lemma \ref{lem_bicomplete}, $(K_\frac{q+2}{2},\pm)$ can be obtained from $(K_q,+)$ by Operation (sb1)-(sb5), and therefore, so does $(G,\sigma)$, a contradiction.

Subcase 2.2: there exist two partite sets $V_j$ and $V_l$ that are not bi-completely adjacent. By a similar argument as in Subcase 1.2, we can deduce that $\{j,l\}=\{1,2\}$,
that every other pair of partite sets are bi-completely adjacent, and that $|V_3|,\ldots,|V_k|\geq 2.$ It follows with the first claim that $V_1\cup V_2$ induces  two vertices with a negative edge between them.

Take a vertices from each of $V_1$ and $V_2$, and two vertices from each of the rest partite sets.
We can see that the signed bi-graph, induced by these vertices, can be obtained from $(K_q,+)$ by adding signed edges and switching at vertices.
Therefore, $(G,\sigma)$ can be obtained from $(K_q,+)$ by Operations (sb1)-(sb5), a contradiction.
\end{proof}

\begin{corollary}
	Every signed graph with chromatic number $q$ can be obtained from $(K_q,+)$ by Operations (sb1)-(sb5).
\end{corollary}

\end{document}